\documentclass[a4paper,10pt]{article}
\usepackage{amsmath,amssymb}
\usepackage[utf8]{inputenc}
\usepackage[colorlinks,citecolor=blue]{hyperref}
\usepackage{graphicx}

\newcommand{\TA}{\mathbb{A}}
\newcommand{\TB}{\mathbb{B}}
\newcommand{\TD}{\mathbb{D}}
\newcommand{\TE}{\mathbb{E}}
\newcommand{\TF}{\mathbb{F}}
\newcommand{\TyH}{\mathbb{H}}
\newcommand{\TI}{\mathbb{I}}

\newcommand{\sph}{\operatorname{Sphere}}

\newtheorem{theorem}{Theorem}[section] 
\newtheorem{proposition}[theorem]{Proposition} 
\newtheorem{conjecture}[theorem]{Conjecture} 
 
\newtheorem{lemma}[theorem]{Lemma}

\newenvironment{proof}{\begin{trivlist}\item{\bf{Proof.}}}
  {\hfill\rule{2mm}{2mm}\end{trivlist}}

\title{A note on gamma triangles \\and local gamma vectors\\
{\small\textit{with an appendix by Alin Bostan}}}
\author{Frédéric Chapoton}
\date{\today}

\begin{document}

\maketitle

When studying simplicial complexes, a basic invariant is the
$f$-vector that counts faces according to their dimensions. It is now
well-known that it is also interesting to consider the $h$-vector,
obtained in a simple way from the $f$-vector. For example, when the
simplicial complex comes from a complete toric fan, the $h$-vector
records the dimensions of the homology groups of the associated toric
variety. Even deeper stands the $\gamma$-vector introduced by Gal
\cite{gal}, which is not really well understood. It can be defined
when the $h$-vector is symmetric, and is conjectured to be nonnegative
under some precise hypotheses involving flagness of simplicial complexes.

Stanley has introduced in \cite{stanley_subdivision} a local
variation on this theme, where the starting point is no longer any
simplicial complex, but rather a subdivision of a standard
simplex. This involves a local $h$-vector, which he proved to be
nonnegative under some technical condition. The local analog of the
$\gamma$-vector was introduced and studied in depth in
\cite{athana_pacific}. It is also conjectured there to be nonnegative under
the appropriate hypothesis of flagness. We refer to the survey article
\cite{athana_survey} for much more details on this beautiful theory.

\smallskip

The author has introduced in \cite{chapoton_SLC51}, motivated by the
study of the combinatorics of simplicial complexes attached to cluster
algebras \cite{FZ2,FZY,reading_speyer}, a finer version of the
$f$-vector, where faces are not only counted according to their
dimension, but in a more refined way using the fact that the
underlying set is split into negative and positive parts. This gives
the $F$-triangle, a polynomial in two variables, that could be
defined for any pure simplicial complex endowed with a prefered
maximal simplex.

Later, an analogue of the $h$-vector in this context, called the
$H$-triangle, has been introduced in \cite{chapoton_plein}. It is
related to the $F$-triangle by a simple birational change of
variables, that extends the classical transformation from the
$f$-vector to the $h$-vector.

The main aim of the present article is to introduce the analogue in
this context of the $\gamma$-vector: we define a $\Gamma$-triangle
starting from the $H$-triangle. To justify that it exists, this
$\Gamma$-triangle is expressed as as a sum of local
$\gamma$-vectors. This implies that the $\Gamma$-triangle is a
refinement of the $\gamma$-vector, that also contains the information of the
local $\gamma$-vector. Conversely, the $\Gamma$-triangle is determined
by the knowledge of all local $\gamma$-vectors of subdivisions of
facets.

We then compute explicitly the $\Gamma$-triangle for all the cluster
simplicial complexes of irreducible Coxeter groups, using the
information on local $\gamma$-vectors for the related subdivisions from
the article \cite{athana_savvidou}.

\medskip

As a general reference on the relationships between the combinatorics
and homology of simplicial complexes and commutative algebra, the
reader may want to consult \cite{stanley_cca}.

As a first side remark, it was observed by Athanasiadis (private
communication) that a formula similar to the relation
\eqref{flat_H_from_G} between $H$-triangle and $\Gamma$-triangle
appears in \cite[Conjectures 1 and 8]{visontai} (see also
\cite[Conjecture 10.2]{branden}), which presents a conjecture of
Gessel about the distribution of descents and inverse descents in the
symmetric groups (two-sided Eulerian polynomials).

As another side remark, let us note that the article
\cite{katz_stapledon} by Katz and Stapledon contains material which
present some formal similarities with the present article, including
versions of $h$-polynomial involving two variables (see their Lemma
5.8). Nevertheless, it seems that both settings cannot be made
identical by any appropriate change of notation.

It may also be interesting to see if the ideas presented here could have
some impact on the study of the zero loci of general $F$-triangles
made in \cite{saito_F_triangle}.

\bigskip

\textit{{\bfseries Acknowledgements:} many thanks to Alin Bostan, who wrote the appendix and thereby allowed the author to complete this article.}


\section{Simplicial complexes and subdivisions}

\label{simplicial_subdivision}

Let $C$ be a finite simplicial complex, which means a collection of
subsets of a fixed finite set, closed under taking subsets. The
elements of $C$ are called faces. The \textit{dimension} of a face of
$C$ is the number of elements in that face minus $1$. Faces of
dimension $0$ are called vertices. The \textit{dimension} of $C$ is
the maximal dimension of the faces of $C$. The simplicial complex $C$
is \textit{pure} if all maximal faces have the same dimension. These
faces of maximal dimension are then called \textit{facets}.

Let us recall briefly the definition of \textit{simplicial
  subdivisions}, see \cite[Section 2]{stanley_subdivision} and
\cite[Section 2.2]{athana_savvidou} for more context and details on
this notion. Let $I$ be a finite set and $2^I$ be the full simplex
with vertex set $I$. A simplicial subdivision of $2^I$ is a simplicial
complex $C_+$ together with a map $\sigma$ from $C_+$ to $2^I$ such
that for every $J \subseteq I$,
\begin{itemize}
\item $\sigma^{-1}(2^J)$ is a subcomplex $C_+(J)$ of $C_+$ which is a
simplicial ball of dimension $|J|-1$ and
\item  $\sigma^{-1}(J)$ consists of the interior faces of $C_+(J)$.
\end{itemize}

The simplicial subdivision $(C_+,\sigma)$ is called \textit{geometric}
if there exists a geometric realisation of $C_+$ (where each face is realised
as an Euclidean simplex) that subdivides geometrically a geometric
realisation of $2^I$.

\medskip

Let us now describe a correspondence between simplicial subdivisions
and some spherical complexes.

Given a simplicial subdivision $C_+$ of $2^I$, one can define a new
simplicial complex $\sph(C_+)$ as follows. Let $(C_+)_0$ be the
underlying set of $C_+$. The underlying set of $\sph(C_+)$ is the
disjoint union $(C_+)_0 \sqcup I$. The faces of $\sph(C_+)$ are pairs
$(F, J)$ (or rather their disjoint union) where $F$ is a face of $C_+$
and $J \subseteq I$ such that $\sigma(F)$ does not intersect $J$.

For a geometric simplicial subdivision $C_+$, the simplicial complex
$\sph(C_+)$ has a geometric realisation that subdivides a sphere, more
precisely the boundary of a cross-polytope.

Conversely, given the spherical simplicial complex $\sph(C_+)$, one
can recover $C_+$ from the knowledge of the distinguished subset of
vertices $I$. Indeed, the faces of $C_+$ are the faces of $\sph(C_+)$
that do not contain any element of $I$.

\medskip

An important class of examples of the situation just described are the
cluster complexes, as appearing in the theory of cluster algebras
\cite{FZ2,FZY}. For every finite Weyl group $W$ and for any choice of
Coxeter element $c \in W$, there is an associated complete fan, called
the cluster fan, whose rays are indexed by almost-positive roots
(positive roots or negative simple roots) in the root system of
$W$. The dual polytope of this fan is a generalized associahedra,
whose edge graph is the flip graph of cluster variables. The
simplicial complex associated to the cluster fan is naturally of the
form $\sph(C_+)$ where $C_+$ is the subcomplex obtained by restriction
to the set of positive roots. The set $I$ is the set of simple roots,
and the structure map $\sigma$ of the subdivision $C_+$ is given by
the support of sets of positive roots.

We will use one important property of the cluster fans, namely the
property that the subcomplexes $C_+(J)$, for $J$ a subset of the set
$I$ of simple roots, is isomorphic to the set $C_+$ for the cluster
fan associated to the parabolic subgroup $W_{I-J}$.

We will also use freely the extension of this theory of cluster fans
to all finite Coxeter groups as done under the name of cambrian fans
by Reading and Speyer in \cite{reading_speyer}. All the properties
that we need do extend to this more general setting.

\subsection{$f$-vector, $h$-vector and $\gamma$-vector}

Let us first recall the definition of the $f$-vector, 
$h$-vector and $\gamma$-vector attached to a simplicial complex.

Given a pure finite simplicial complex $C$ of dimension $d-1$, its $f$-vector
is the sequence of integers $(f_{-1},f_0,\dots,f_{d-1})$ where $f_i$ is
the number of faces with $i+1$ vertices in $C$. The associated
$f$-polynomial is defined as
\begin{equation}
  f_C(x) = \sum_{0\leq i \leq d} f_{i-1} x^i.
\end{equation}

The $h$-polynomial of the simplicial complex $C$ is then defined as
\begin{equation}
  \label{h_from_f}
  h_C(x) = (1-x)^{d} f_C\left(\frac{x}{1-x}\right) = \sum_{0\leq i \leq d} f_{i-1} x^i(1-x)^{d-i}.
\end{equation}
The reverse transformation is given by
\begin{equation}
  \label{f_from_h}
  f_C(x) = (1+x)^{d} h_C\left(\frac{x}{1+x}\right).
\end{equation}

When the $h$-polynomial is written as
\begin{equation}
  h_C(x) = \sum_{0\leq i \leq d} h_i x^i,
\end{equation}
the sequence of integers $(h_0,h_1,\dots,h_{d})$ is called the $h$-vector of $C$.

When $C$ is an homology sphere, the $h$-vector satisfies the symmetry
property $h_i = h_{d-i}$ for all $0 \leq i \leq d$. In this case, one can always write
\begin{equation}
  \label{h-gamma}
  h(x) = \sum_{0 \leq i \leq d/2} \gamma_i x^i (1+x)^{d-2i},
\end{equation}
for some uniquely defined integer coefficients $\gamma_i$. These coefficients form the
$\gamma$-vector attached to the simplicial complex $C$.

There is a famous conjecture of Gal about these coefficients \cite{gal}. Recall
that a simplicial complex is said to be \textit{flag} if all minimal
non-faces have two elements.

\begin{conjecture}[{\cite[conjecture 2.1.7]{gal}}]
  The $\gamma$-vector has nonnegative coordinates for every
  flag homology sphere.
\end{conjecture}

\subsection{Local $f$-vector, $h$-vector and $\gamma$-vector}

Let us now recall the definition of the local $f$-vector, local
$h$-vector and local $\gamma$-vector attached to a simplicial
subdivision.

Let $I$ be a finite set of cardinality $d$ and let $C_+$ be a
simplicial subdivision of the simplex $2^I$. The local $h$-polynomial
$h^{\ell}_{C_+}(x)$ is the alternating sum of the $h$-polynomials of the
restrictions of $C_+$ to the faces of $2^I$. More precisely,
\begin{equation}
  \label{hloc_from_h}
  h^{\ell}_{C_+}(x) = \sum_{J \subseteq I} (-1)^{|I-J|} h_{C_+(J)}(x).
\end{equation}
Conversely, by Möbius inversion on the boolean lattice of subsets,
\begin{equation}
  \label{h_from_hloc}
  h_{C_+}(x) = \sum_{J \subseteq I} h^\ell_{C_+(F)}(x).
\end{equation}
When expanded as
\begin{equation}
  h^{\ell}_{C_+}(x) = \sum_{i=0}^{n} h^\ell_i x^i,
\end{equation}
the sequence $(h^\ell_0,\dots,h^\ell_d)$ is called the local $h$-vector of $C_+$.

The local $h$-vector is known to be symmetric ($h^\ell_i = h^\ell_{d-i}$ for all
$0\leq i\leq d$) for any simplicial subdivision and nonnegative for every
geometric simplicial subdivision \cite{stanley_subdivision}.

Because of this symmetry property, one can define the local
$\gamma$-vector in the same way as the $\gamma$-vector was defined
from the $h$-vector. Namely, one can always write
\begin{equation}
  \label{h-gamma-local}
  h^\ell(x) = \sum_{0 \leq i \leq d/2} \gamma^\ell_i x^i (1+x)^{d-2i},
\end{equation}
for some uniquely defined integer coefficients $\gamma^\ell_i$. These
coefficients form the local $\gamma$-vector attached to the simplicial
subdivision $C_+$.

The local $\gamma$-polynomial is multiplicative for the natural join operation on
simplicial subdivisions, see \cite[Lemma 2.2]{athana_savvidou}.

\section{$F$-triangle, $H$-triangle}

Let us now recall the definition of $F$-triangles and $H$-triangles,
originally introduced in \cite{chapoton_SLC51} in the context of cluster
complexes. They were later related to a third polynomial, the
$M$-triangle, that will not be considered here.

Let $C$ be a pure finite spherical simplicial complex of dimension
$d-1$, with a distinguished facet $T$. The $F$-triangle of the pair
$(C,T)$ is the generating polynomial
\begin{equation}
  F_{C,T}(x,y) = \sum_{0\leq i,j\leq d} F_{i,j} x^i y^j,
\end{equation}
where $F_{i,j}$ is the number of faces of $C$ of cardinality $i+j$
that are made of $i$ elements not in $T$ and $j$ elements in $T$. When
setting $y=x$, this reduces to the usual $f$-polynomial, that is
$F_{C,T}(x,x) = f_C(x)$.

The $H$-triangle is then defined as 
\begin{equation}
  \label{H_from_F}
  H_{C,T}(x,y) = (1-x)^d F_{C,T}\left(\frac{x}{1-x}, \frac{x y}{1-x}\right).
\end{equation}
The reverse conversion formula is
\begin{equation}
  \label{F_from_H}
  F_{C,T}(x,y) = (1+x)^d H_{C,T}\left(\frac{x}{1+x}, \frac{y}{x}\right).
\end{equation}
When setting $y=1$ in the $H$-triangle, one gets back the usual
$h$-polynomial of the simplicial complex $C$, that is
$H_{C,T}(x,1) = h_C(x)$. The conversion formulas also extend the
usual ones between $f$-vectors and $h$-vectors.

\medskip

In this article, we will only consider $F$-triangles and $H$-triangles
in the case where $C = \sph(C_+)$ for some simplicial subdivision
$C_+$ of $2^I$, taking as distinguished facet $T$ the unique facet of
$\sph(C_+)$ with set of vertices $I$.

\section{$\Gamma$-triangle}

Let us now introduce the main novelty of the article, the
$\Gamma$-triangle. It is closely related to the $F$-triangle and
$H$-triangle, and can be seen as a condensed way to describe these
polynomials, with half less coefficients.

Consider a simplicial sphere of the form $\sph(C_+)$. Assume for the moment that one can write its $H$-triangle in the following shape
\begin{equation}
  \label{H_from_G}
  H(x,y) = (1+x)^d \sum_{\substack{0\leq i\\ 0\leq j \leq d-2i}} \gamma_{i,j} \left(\frac{x}{(1+x)^2}\right)^i \left(\frac{1+x y}{1+x}\right)^j,
\end{equation}
for some integer coefficients $\gamma_{i,j}$. If this is possible,
there is a unique way to do so. The coefficients $\gamma_{i,j}$ are called the $\Gamma$-triangle of $\sph(C_+)$. Note that the coefficients fit inside
a triangle, whence the name.

We will prove in the next section that this decomposition is always
possible in this context, and give an expression for the coefficients
$\gamma_{i,j}$ in terms of the local $\gamma$-vectors for
sub-complexes of $C_+$.

The formula \eqref{H_from_G} can also be displayed as
\begin{equation}
  \label{flat_H_from_G}
  H(x,y) = \sum_{i,j} \gamma_{i,j} x^i (1+x y)^j (1+x)^{d-2i-j}.
\end{equation}
This reduces to the usual formula \eqref{h-gamma} for the
$\gamma$-vector when $y=1$. This implies that a necessary condition
for \eqref{H_from_G} to exist is the symmetry of the $h$-vector. It
also implies that the $\Gamma$-triangle is a refinement of the
$\gamma$-vector, in the sense that
\begin{equation}
  \gamma_i = \sum_j \gamma_{i,j}.
\end{equation}

Using the relation \eqref{F_from_H} between
$F$-triangle and $H$-triangle, one obtains that
\begin{equation}
  F(x,y) = (1+2x)^d \sum_{i,j} \gamma_{i,j} \left(\frac{x(1+x)}{(1+2x)^2}\right)^i \left(\frac{1+x+y}{1+2x}\right)^j.
\end{equation}

\medskip

As a simple concrete example, let us consider for $n\geq 4$ the
regular $n$-polygon, whose $F$, $H$ and $\Gamma$-triangles (with
respect to any facet) are
\begin{equation}
  \label{polygon}
\left(\begin{array}{rrr}
1 &  & \\
2 & 2 & \\
1 & n - 2 & n-3
\end{array}\right),
\left(\begin{array}{rrr}
 &  & 1\\
 & 2 & 0\\
1 & n - 4 &0
\end{array}\right),
\left(\begin{array}{rrr}
1 & \\
0 & \\
0 & n - 4
\end{array}\right).
\end{equation}
Here and after, the coefficients are displayed with the power of $x$ (index $i$) increasing from left to right and the power of $y$ (index $j$) increasing from bottom to top.

\begin{figure}\centering
  \includegraphics[width=3cm]{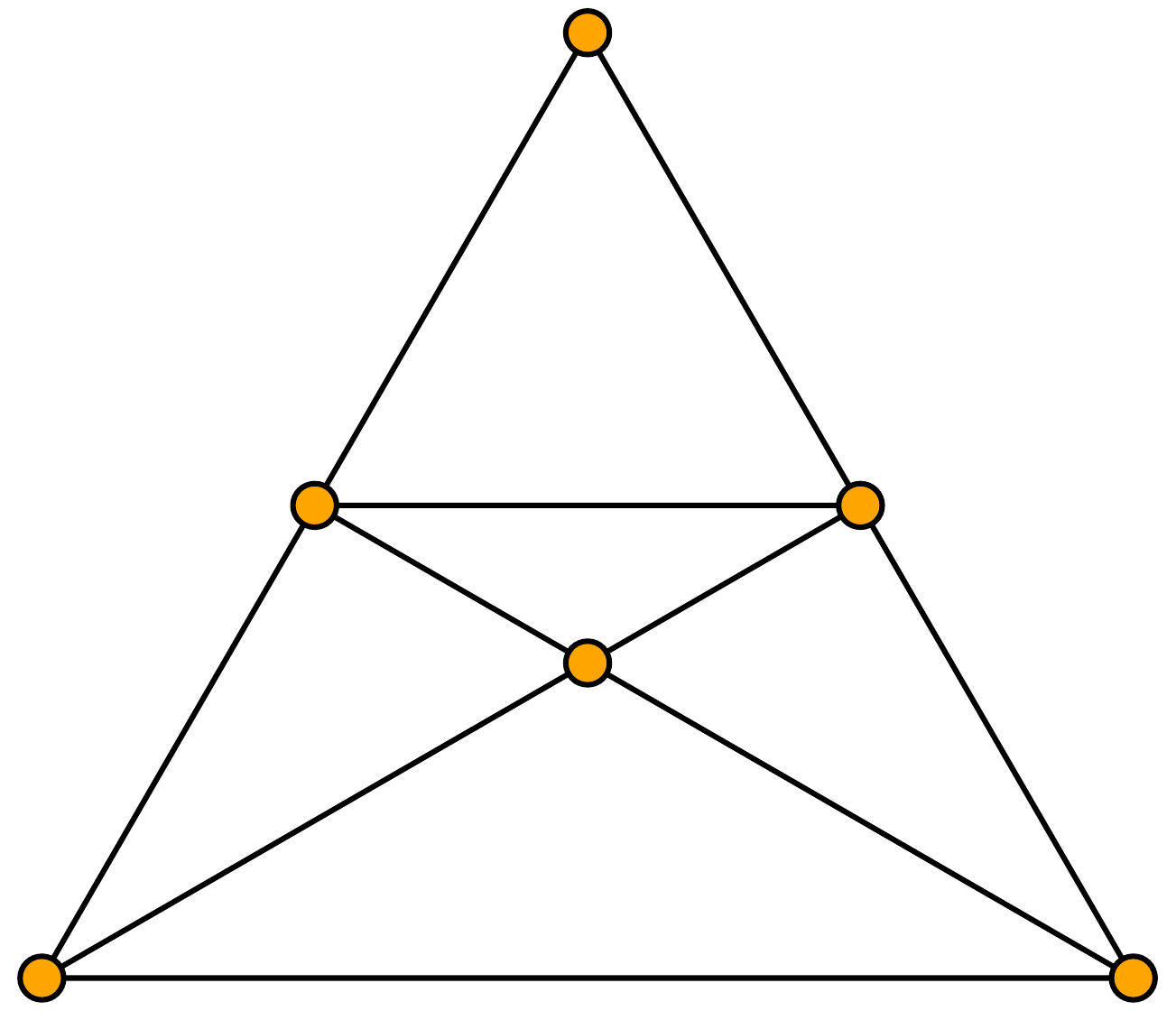}
  \caption{Positive part $C_+$ of a cluster complex of type $\TA_3$.}
  \label{fig:A3}
\end{figure}

Here is another example, for the cluster complex of type $\TA_3$ whose positive
part $C_+$ is depicted in figure \ref{fig:A3}. Here, the $F$, $H$ and
$\Gamma$-triangles of $\sph(C_+)$ are
\begin{equation}
\left(\begin{array}{rrrr}
1 &  & &\\
3 & 3 & &\\
3 & 8 & 5&\\
1 & 6 & 10& 5
\end{array}\right),
\left(\begin{array}{rrrrr}
 &   & &1\\
 &  & 3 & 0\\
 & 3 & 2& 0\\
1 & 3 & 1& 0
\end{array}\right),
\left(\begin{array}{rr}
1 &   \\
0 &  \\
0 & 2 \\
0 & 1 
\end{array}\right).
\end{equation}

\subsection{Existence and local description of $\Gamma$-triangle}

We will proceed here to a computation, ending with a formula that
implies the existence of the $\Gamma$-triangle for $\sph(C_+)$,
together with an expression for the $\gamma_{i,j}$ coefficients in
terms of all the local $\gamma$-vectors for the subcomplexes $C_+(J)$.

Let us start by the equality
\begin{equation}
  F_{\sph(C_+)}(x,y) = \sum_{J \subseteq I} \sum_{f \in C_+(I-J)} x^{|f|} y^{|J|}
\end{equation}
holding by definition of the $F$-triangle and by the description of
the faces of $\sph(C_+)$ from the faces of $C_+$. Using $f$-vectors, one gets
\begin{equation}
  F_{\sph(C_+)}(x,y) = \sum_{J \subseteq I} y^{|J|} f_{C_+(I-J)}(x).
\end{equation}

Passing to the $H$-triangle by \eqref{H_from_F}, one gets
\begin{equation}
  H_{\sph(C_+)}(x,y) = (1-x)^{|I|} \sum_{J \subseteq I} \left(\frac{x y}{1-x}\right)^{|J|} f_{C_+(I-J)}\left(\frac{x}{1-x}\right).
\end{equation}

Expressing $f$-vectors in terms of $h$-vectors by \eqref{h_from_f}, this becomes
\begin{equation}
  H_{\sph(C_+)}(x,y) = (1-x)^{|I|} \sum_{J \subseteq I} \left(\frac{x y}{1-x}\right)^{|J|} (1-x)^{-|I-J|} h_{C_+(I-J)}(x)
\end{equation}
which is just
\begin{equation}
  H_{\sph(C_+)}(x,y) = \sum_{J \subseteq I} (x y)^{|J|} h_{C_+(I-J)}(x).
\end{equation}

Expressing $h$-vectors in terms of local $h$-vectors by \eqref{h_from_hloc}, one gets
\begin{equation}
  H_{\sph(C_+)}(x,y) = \sum_{J \subseteq I} (x y)^{J} \sum_{K \subseteq I-J} h^\ell_{C_+(K)}(x).
\end{equation}
which can be rewritten
\begin{equation}
  H_{\sph(C_+)}(x,y) = \sum_{K \subseteq I} (1+x y)^{|I-K|} h^\ell_{C_+(K)}(x).
\end{equation}

Then passing from local $h$-vectors to local $\gamma$-vectors by \eqref{h-gamma-local}, one gets
\begin{equation}
  H_{\sph(C_+)}(x,y) = \sum_{K \subseteq I} (1+x y)^{|I-K|} \sum_i \gamma^\ell_{C_+(K),i} \left(\frac{x}{(1+x)^2}\right)^i (1+x)^{|K|}.
\end{equation}

By comparing carefully with the desired expression \eqref{H_from_G},
one finds at last the following expression for the coefficients
$\gamma_{i,j}$.
\begin{proposition}
  \label{relation_Gamma_gamma}
  The $\Gamma$-triangle $\Gamma(x,y)$ of $\sph(C_+)$ can be expressed as
  \begin{equation}
    \Gamma(x,y) = \sum_{i,j} \gamma_{i,j} x^i y^j = \sum_{K \subseteq I} \gamma^\ell_{C_+(K)}(x) y^{|I-K|},
  \end{equation}
  where the sum is over all subsets $K$ of the set $I$.
\end{proposition}

This implies that Athanasiadis' version of the Gal's conjecture for
the local $\gamma$-vectors (\cite[Conjecture 3.6]{athana_survey})
would imply the following version of Gal's conjecture for
$\Gamma$-triangles.

Let $\sph(C_+)$ be a flag spherical simplicial complex of the type defined in
section \ref{simplicial_subdivision}.
\begin{conjecture}
  The $\Gamma$-triangle for the pair $(\sph(C_+),I)$ has nonnegative
  coefficients.
\end{conjecture}


This holds true for all cluster complexes of finite type, see the next
section and tables at the end of the article.

As an example of what happens without flagness, let us consider the
case of the triangle made of $3$ edges, and its $F$-triangle with
respect to one edge. The $H$-triangle and $\Gamma$-triangle were
already computed in \eqref{polygon} for $n=3$, but the $\Gamma$-triangle
has a negative coefficient.

\section{Explicit values for finite cluster fans}

In this section, we compute the $\Gamma$-triangle for cluster fans in
types $\TA,\TB,\TD$.

\smallskip

Using Prop. \ref{relation_Gamma_gamma}, the $\Gamma$-triangle of the cluster
fan of a given Dynkin diagram $\Phi$, seen as a polynomial
$\Gamma(x,y)$, is therefore determined by
\begin{equation}
  \label{Gamma_dynkin}
  \Gamma_{\Phi}(x,y) = \sum_{J \subset I} \gamma^{\ell}_{I-J}(x) y^{|J|},
\end{equation}
where $J$ is the set of vertices not in the subdiagram.

We will use this formula in the next sections, in the special cases of
types $\TA$, $\TB$ and $\TD$, to obtain algebraic equations for the
generating series of $\Gamma$-triangles. We also give explicit
expressions for coefficients of these
generating series.

\smallskip

Note that the coefficients $\gamma_{1,i}$ for cluster fans have an
explicit description as the numbers of non-simple roots, in the root
system of $W$, according to the size of their support. This follows
directly from Remark 1 in \cite[\S 5]{athana_savvidou} and
prop. \ref{relation_Gamma_gamma}.

\subsection{Type $\TA$}

Let us start by a computation in type $\TA$.

\begin{proposition}
  The coefficient of $x^k y^\ell$ in the $\Gamma$-triangle of the
  associahedra of type $\TA_n$ is
  \begin{equation}
    \frac{\ell+1}{n-k+1}\binom{n}{k}\binom{n-k-\ell-1}{k-1}
  \end{equation}
  for $k \geq 0$, $\ell\geq 0$ and $\ell+2k \leq n$. 
\end{proposition}

The proof follows.

Let $g_{\TA}$ be the following generating series
\begin{equation}
  \sum_{m\geq 0}\sum_{k\geq 0} \frac{1}{k+m+1}\binom{2k+m}{k}\binom{k+m-1}{k-1}x^k t^{2k+m}.
\end{equation}
This is the generating series for the known local $\gamma$-vectors for
the positive part of the cluster complex of type $\TA$ , see
\cite[proposition 3.1, equation (8)]{athana_savvidou}.

Let $G_{\TA}$ be the following generating series
\begin{equation}
  \sum_{m\geq 0}\sum_{k\geq 0}\sum_{\ell\geq 0} \frac{\ell+1}{\ell+k+m+1}\binom{\ell+2k+m}{k}\binom{k+m-1}{k-1}x^k y^\ell t^{2k+m+\ell} .
\end{equation}
This is the generating series for the expected $\Gamma$-triangles for
the cluster complex of type $\TA$.

\begin{proposition}
  We have the following relation:
  \begin{equation}
    \label{petitA_grandA}
    G_{\TA} = g_{\TA} + y t \, g_{\TA} G_{\TA}.
  \end{equation}
\end{proposition}
\begin{proof}
  Let us compute the coefficient of $(x t^2)^k (y t)^\ell t^m$ in $g_{\TA} G_{\TA}$.
  This is given by the finite sum
  \begin{multline*}
    \sum_{\substack{k_1+k_2 = k\\m_1+m_2=m}}
    \frac{1}{k_1+m_1+1}\binom{2k_1+m_1}{k_1}\binom{k_1+m_1-1}{k_1-1}\\
    \frac{\ell+1}{\ell+k_2+m_2+1}\binom{\ell+2k_2+m_2}{k_2}\binom{k_2+m_2-1}{k_2-1}
  \end{multline*}
  with $k_1\geq 0$, $k_2\geq 0$, $m_1\geq 0$ and $m_2\geq 0$.
  This expression can be rewritten as
  \begin{multline*}
    \sum_{\substack{k_1+k_2 = k\\m_1+m_2=m}}
    \frac{k_1}{(2k_1+m_1+1)(k_1+m_1)}\binom{2k_1+m_1+1}{k_1}\binom{k_1+m_1}{m_1}\\
    \frac{(\ell+1)k_2}{(2k_2+m_2+\ell+1)(k_2+m_2)}\binom{2k_2+m_2+\ell+1}{k_2}\binom{k_2+m_2}{m_2}
  \end{multline*}
  By a summation formula of Carlitz \cite[Th. 6 (5.14)]{carlitz}, this is equal to
  \begin{equation}
    \frac{(\ell+2)k}{(2k+m+\ell+2)(k+m)}\binom{2k+m+\ell+2}{k}\binom{k+m}{m}
  \end{equation}
  which is exactly the coefficient of $(x t^2)^k (y t)^{\ell+1} t^m$ in
  $G_{\TA} - g_{\TA}$.
\end{proof}

But the equation \eqref{petitA_grandA} is exactly the relation given
by applying \eqref{Gamma_dynkin} in type $\TA$, where Dynkin diagrams
are line-shaped graphs. Namely, either $J$ is empty and the subdiagram
on $I-J$ is the full diagram, or there exists a leftmost vertex that
is in $J$. The first case correspond to the term $g_{\TA}$ in the
right hand side of \eqref{petitA_grandA}. In the second case, one can
use the multiplicativity of local $\gamma$-vectors to separate the
leftmost connected component of $I-J$. This gives the second term in
the right hand side of \eqref{petitA_grandA}.

\subsection{Type $\TB$}

Let us proceed to the similar computation in type $\TB$.

\begin{proposition}
  The coefficient of $x^k y^\ell$ in the $\Gamma$-triangle of the
  associahedra of type $\TB_n$ is
  \begin{equation}
    \binom{n}{k}\binom{n-k-\ell-1}{k-1},
  \end{equation}
  for $k \geq 0$, $\ell\geq 0$ and $\ell+2k \leq n$. 
\end{proposition}

The proof, similar to the case of type $\TA$, follows.

Let $g_{\TB}$ be the following generating series
\begin{equation}
  \sum_{m\geq 0}\sum_{k\geq 0} \binom{2k+m}{k}\binom{k+m-1}{k-1}x^k t^{2k+m}.
\end{equation}
This is the generating series for the known local $\gamma$-vectors for type $\TB$, see \cite[proposition 3.2, equation (12)]{athana_savvidou}.

Let $G_{\TB}$ be the following generating series
\begin{equation}
  \sum_{m\geq 0}\sum_{k\geq 0}\sum_{\ell\geq 0} \binom{2k+\ell+m}{k}\binom{k+m-1}{k-1}x^k y^\ell t^{2k+\ell+m}.
\end{equation}
This is the generating series for the expected $\Gamma$-triangles for type $\TB$.

\begin{proposition}
  We have the following relations:
  \begin{equation}
    \label{petitB_grandB}
    G_{\TB} = g_{\TB} + y t \, g_{\TA} G_{\TB} \quad\text{and}\quad
    G_{\TB} = g_{\TB} + y t \, g_{\TB} G_{\TA}.
  \end{equation}
\end{proposition}
\begin{proof}
  The proof of these two equations is very similar. Let us give some
  details only for the first one. One computes the coefficient of
  $(x t^2)^k (y t)^\ell t^m$ in $g_{\TA} G_{\TB}$.  This is given by
  \begin{multline*}
    \sum_{\substack{k_1+k_2 = k\\m_1+m_2=m}}
    \frac{1}{k_1+m_1+1}\binom{2k_1+m_1}{k_1}\binom{k_1+m_1-1}{k_1-1}\\
    \binom{\ell+2k_2+m_2}{k_2}\binom{k_2+m_2 -1}{k_2-1},
  \end{multline*}
  which can be rewritten
  \begin{multline*}
    \sum_{\substack{k_1+k_2 = k\\m_1+m_2=m}}
    \frac{k_1}{(2k_1+m_1+1)(k_1+m_1)}\binom{2k_1+m_1+1}{k_1}\binom{k_1+m_1}{m_1}\\
    \binom{2k_2+m_2+\ell}{k_2}\binom{k_2+m_2-1}{m_2}.
  \end{multline*}
  By applying \cite[Th. 6 (5.15)]{carlitz} (with the correct right-hand side that involves $+c n$), one gets that this is equal to
  \begin{equation}
    \binom{2k +m+\ell+1}{k}\binom{k+m-1}{m},
  \end{equation}
  which is readily seen to be the coefficient of
   $(x t^2)^k (y t)^{\ell+1} t^m$ in
  $G_{\TB} - g_{\TB}$.
\end{proof}

By the same proof as in type $\TA$, the equations
\eqref{petitB_grandB} are exactly the relations given by
\eqref{Gamma_dynkin} between the local $\gamma$-vectors
and the $\Gamma$-triangle in type $\TB$.

\subsection{Type $\TD$}

There is an amusing and unexpected relation between the
$\Gamma$-triangles of cluster fans of type $\TB$ and $\TD$.

\begin{proposition}
  For every $n\geq 3$, the $\Gamma$-triangle of type $\TD_n$ is
  obtained from the $\Gamma$-triangle of type $\TB_{n-1}$ by adding a
  bottom line, which is the local $\gamma$-vector of type $\TD_n$.
\end{proposition}
\begin{proof}
  This follows from the statements below.
\end{proof}

Let $g_{\TD}$ be the following generating series
\begin{equation}
  \sum_{m\geq 0}\sum_{k\geq 1} \frac{2k+m-2}{k}\binom{2k-2}{k-1}\binom{2k+m-2}{2k-2} x^k t^{2k+m}.
\end{equation}
This is the generating series for the known local $\gamma$-vectors for type $\TD$ see \cite[proposition 3.3]{athana_savvidou}.

Let $G_{\TD}$ be the generating series for the expected $\Gamma$-triangles for type $\TD$ for $n\geq 2$. The proposition above is equivalent to
\begin{equation}
  \label{bizarre_B_D}
  G_{\TD} = y t (G_{\TB} - 1) + g_{\TD},
\end{equation}
which is therefore what we want to prove.

\begin{proposition}
  We have the following relation:
  \begin{equation}
    \label{petitD_grandD}
    G_{\TD} = g_{\TD} + 2 y t (g_{\TA}-1) + (y t)^2 g_{\TA} + y t g_{\TA} G_{\TD}.
  \end{equation}
\end{proposition}
\begin{proof}
  This is the consequence of the general relation
  \eqref{Gamma_dynkin} from $\gamma$-vector to
  $\Gamma$-triangle. If the subset $J$ is empty, we get the first
  term. If the subset $J$ is made of one of the two vertices in the
  fork of the $\TD$ Dynkin diagram, we get the next term. If $J$ is
  made of both, then we get the third term. Otherwise, $J$ has at
  least one element on the tail part of the Dynkin diagram, and one
  can cut at the farthest one from the forking point.
\end{proof}

To deduce \eqref{bizarre_B_D} from \eqref{petitD_grandD}, it is enough to prove the following.
\begin{lemma}
  \begin{equation}
    \label{mini_D}
    g_{\TB} - 1 =  2 (g_{\TA}-1) + g_{\TA}g_{\TD}.
  \end{equation}
\end{lemma}
\begin{proof}
  This follows from the algebraicity of these $3$ series, and more
  precisely from the equations that they satisfy, see the appendix
  \ref{appendix} for a detailed proof.
\end{proof}

Indeed, by taking the sum of \eqref{petitD_grandD} and $yt$ times \eqref{mini_D}, one gets
\begin{equation}
  (G_{\TD} - g_{\TD})(1-y t g_{\TA}) = y t (g_{\TB} - 1 + y t g_{\TA}).
\end{equation}
But one can deduce from \eqref{petitB_grandB} that
\begin{equation}
  (G_{\TB} - 1)(1-y t g_{\TA}) = g_{\TB} - 1 + y t g_{\TA}.
\end{equation}
Comparing these two equations implies \eqref{bizarre_B_D}.

\smallskip

For example, here are the $\Gamma$-triangles of type $\TB_5$ and type $\TD_6$ :
\begin{equation*}
\left(\begin{array}{rrr}
1 &&\\
0&&\\
0 &5&\\
0 &5&\\
0 &5 &10\\
0 &5 &20
\end{array}\right),
\left(\begin{array}{rrrr}
1 &&&\\
0&&&\\
0 &5&&\\
0 &5&&\\
0 &5 &10&\\
0 &5 &20&\\
0 &4 &24 &8
\end{array}\right).
\end{equation*}

\section{Quadrangulations and dissections}

Another large class of pure flag spherical simplicial complexes of the
shape $\sph(C_+)$ is obtained from Stokes complexes associated to
quadrangulations \cite{chapoton_stokes} and more general dissection
complexes as introduced by Garver and McConville in \cite{garver_mcconville} and
further studied and extended in \cite{pilaud_manneville, ppp}. We will only briefly mention
some interesting examples, with no proof.

Among these simplicial complexes, one can find families indexed by $n$
that should be the simplest possible, in the sense that the dual
polytopes are not products of simpler cases and have as few vertices
as possible for a given dimension $n$.

For quadrangulations and Stokes complexes, there is a family having
vertices enumerated by the Lucas numbers (OEIS \href{http://oeis.org/A000032}{A32}), see \cite[\S 4.1]{chapoton_stokes}. The $\Gamma$-triangles for this family satisfy the linear recursion
\begin{equation}
u_{n+1} = (y^2 + 2 x) u_n - x^2 u_{n-1} + x y u_{n-1},
\end{equation}
with initial values $0, 1$.


For general dissections, there is a family having their vertices
counted by the Pell numbers (OEIS \href{http://oeis.org/A000129}{A129}). This seems to be closely related to the
objects considered in \cite{law}. The $\Gamma$-triangles for this
family satisfy the linear recursion
\begin{equation}
u_{n+1} = y u_n + x u_{n-1},
\end{equation}
with initial values $0, 1$. These are therefore some kind of Fibonacci polynomials. Strangely, the discriminant of this recursion is the $\Gamma$-triangle for $\TI_2(6)$.

\section{Tables of $\Gamma$-triangles}

\label{tables}

Types of rank $2$ and $3$ :
\begin{equation*}
\left(\begin{array}{rr}
1 & \\
0 & \\
0 & h - 2
\end{array}\right),
\left(\begin{array}{rr}
1 & \\
0 & \\
0 & \frac{6(h - 2)}{h+2}\\
0 & \frac{3(h - 2)^2}{2(h+2)}
\end{array}\right).
\end{equation*}
For rank $2$, the parameter $h \geq 2$ is the Coxeter number,
corresponding to $\TI_2(h)$. For rank $3$, the parameter is also the
Coxeter number $h$, with possible values $2,4,6,10$ corresponding to
$\TA_1^3, \TA_3, \TB_3$ and $\TyH_3$.

\medskip

Types $\TA_4$, $\TB_4$, $\TD_4$, $\TF_4$ and $\TyH_4$:
\begin{equation*}
\left(\begin{array}{rrr}
1 &  &  \\
0 &  &  \\
0 & 3 &  \\
0 & 2 &  \\
0 & 1 & 2
\end{array}\right),
\left(\begin{array}{rrr}
1 &  &  \\
0 &  &  \\
0 & 4 &  \\
0 & 4 &  \\
0 & 4 & 6
\end{array}\right),
\left(\begin{array}{rrr}
1 &  &  \\
0 &  &  \\
0 & 3 &  \\
0 & 3 &  \\
0 & 2 & 2
\end{array}\right),
\left(\begin{array}{rrr}
1 &  &  \\
0 &  &  \\
0 & 4 &  \\
0 & 6 &  \\
0 & 10 & 9
\end{array}\right),
\left(\begin{array}{rrr}
1 &  &  \\
0 &  &  \\
0 & 5 &  \\
0 & 9 &  \\
0 & 42 & 40
\end{array}\right).
\end{equation*}

\medskip

Types $\TE_6$, $\TE_7$ and $\TE_8$:
\begin{equation*}
\left(\begin{array}{rrrr}
1 & & &\\
0 & & &\\
0 & 5 & &\\
0 & 5 & &\\
0 & 6 & 11 &\\
0 & 7 & 23 &\\
0 & 7 & 35 & 13
\end{array}\right),
\left(\begin{array}{rrrr}
1 & & &\\
0 & & &\\
0 & 6 & &\\
0 & 6 & &\\
0 & 7 & 16 &\\
0 & 9 & 36 &\\
0 & 12 & 69 & 28\\
0 & 16 &124 & 112
\end{array}\right),
\left(\begin{array}{rrrrr}
1 & & & &\\
0 & & & &\\
0 & 7 & & &\\
0 & 7 & & &\\
0 & 8 & 22 & &\\
0 & 10 & 48 & &\\
0 & 14 & 94 & 46 &\\
0 & 22 &192 & 194 &\\
0 & 44 &484 & 784 & 120
\end{array}\right).
\end{equation*}

\section{Appendix}

\label{appendix}

Recall that $g_{\TA}, g_{\TB}, g_{\TD}$ are power series in $\mathbb{Q}[x][[t]]$ defined by
\[g_{\TA} := 1+\sum _{n=1}^{\infty } \left( \sum _{i=1}^{\left\lfloor n/2
\right\rfloor}{\frac {{n\choose i}{n-i-1\choose i-1}}{n-i+1}} {x}^{i}\right){t
}^{n}
= 1+x{t}^{2}+x{t}^{3}+ \left( 2\,{x}^{2}+x \right) {t}^{4}+ 
\cdots ,\]
\[g_{\TB} := 1+\sum _{n=1}^{\infty } \left( \sum _{i=1}^{\left\lfloor n/2
\right\rfloor}{n\choose i}{n-i-1\choose i-1}{x}^{i} \right){t}^{n}
=1+2\,x{t}^{2}+3\,x{t}^{3}+ \left( 6\,{x}^{2}+4\,x \right) {t}^{4}+
\cdots
\]
and
\[g_{\TD} := \sum _{n=0}^{\infty } \left( \sum _{i=1}^{\left\lfloor n/2
\right\rfloor}{\frac {{2\,i-2\choose i-1}{n-2\choose 2\,i-2} \left( n-
2 \right) }{i}} {x}^{i}\right) {t}^{n}
=x{t}^{3}+ \left( 2\,{x}^{2}+2\,x \right) {t}^{4}+ 
\cdots
.\]

We prove here that these power series are algebraic, and that they satisfy the
following identity: 
\[g_{\TB}-1 = 2(g_{\TA}-1) + g_{\TA} g_{\TD}.\]

To do this, we introduce an auxiliary algebraic power series in
$\mathbb{Q}[x][[t]]$,
\[ g = \sqrt { \left( 1-t \right) ^{2}-4\,x{t}^{2}} 
= 1-t-2\,x{t}^{2}-2\,x{t}^{3} - \left( 2\,{x}^{2}+2\,x \right) {t}^{4}
+\cdots
\]

We claim that the following relations hold in $\mathbb{Q}[[x,t]]$:
\begin{align}
	g_{\TA} & =  \frac {1+t-g}{2 t \left( tx+1 \right) },  \label{conjA}\\
	g_{\TB} & =  \frac {2\,tx+g-t+1}{2\, g\, \left( tx+1 \right) }, \label{conjB}\\
	g_{\TD} & =  \frac {\left( g-1 \right)  \left( g-1+t \right) }{2\, g}. \label{conjD}
\end{align}

Assuming these identities, it is immediate to check our claims, namely that 
$g_{\TA}, g_{\TB}, g_{\TD}$ are algebraic, and related by
\[g_{\TB}-1-2(g_{\TA}-1) - g_{\TA} g_{\TD} = 
\frac { \left( 1+g \right)  \left( {g}^{2}-
 \left( 1-t \right) ^{2}+4\,x{t}^{2} \right) }{4 \, t \left( tx+1 \right) g}
= 0.
\]

It is therefore enough to prove identities~\eqref{conjA}, \eqref{conjB}
and~\eqref{conjD}.

First, let us remark that $g$ is equal to 
\[
g = (1-t)\sqrt{1 - 4x \left(\frac{t}{1-t}\right)^2}.
\]

Setting $u = x t^2/(1-t)^2$ in the binomial formula
\[
\sqrt{1-4u} = 1- 2\, \sum _{i=2}^{\infty }{\frac {{2\,i-2\choose i-1}}{i}}{u}^{i}
\]
yields the expansion
\[
g = 1-t-2\, \sum _{i\geq 1}{\frac {{2\,i-2\choose i-1}}{i
}} x^i \frac{t^{2i}}{(1-t)^{2i-1}}.
\]
Combining this with the classical expansion 
\[
\frac{1}{\left( 1-t \right) ^{2\,i-1}} =\sum _{n=0}^{
\infty }{n+2\,i-2\choose {2\,i-2}}{t}^{n}\]
proves that
\[
g = 1-t-2\,\sum _{i \geq 1} \left( \sum _{n \geq 0}{\frac {{2\,i-2\choose i-1}{n+2i-2\choose 2\,i-2}}{i
}} t^{n+2i} \right){x}^{i},
\]
in other words:
\begin{equation}\label{eq:C}
g = 1-t-2\,\sum _{n=1}^{\infty } \left( \sum _{i=1}^{\left\lfloor n/2
\right\rfloor}{\frac {{2\,i-2\choose i-1}{n-2\choose 2\,i-2}}{i
}} x^i \right){t}^{n}.
\end{equation}

Now the proof of identity~\eqref{conjA} amounts to a direct verification,
based on the binomial identity 
\[
{\frac {{n-2\choose i-1}{n-i-2\choose i-2}+{n-1\choose i}{n-i-2
\choose i-1}}{n-i}}={\frac {{2\,i-2\choose i-1}{n-2\choose 2\,i-2}}{i}
} \quad \text{for} \; n\geq 2 \; \text{and} \; 1 \leq i \leq  \lfloor n/2 \rfloor.
\]

Identity~\eqref{conjD} follows from~\eqref{eq:C} and the definition of $g_{\TD}$, by using the following observation, where $\theta_t = t \, {\frac {\partial}{{\partial}t}}$ stands for the Euler derivation:
\[g_{\TD} = \left( 2 - \theta_t  \right) \left( \frac{g - 1 + t}{2} \right).\]

Indeed, proving~\eqref{conjD} amounts to checking that $g$ satisfies the differential equation
\[t g  \frac{\partial g}{\partial t}  -
 g ^{2}-t+1 = 0,\]
which is obvious from the definition of $g$.

Finally, identity~\eqref{conjB} follows from~\eqref{conjA} using the following 
observation:
\[
g_{\TB}(x,t) =
\frac{ \partial \left( t \, g_{\TA}(x/t,t) \right)} {\partial t} (xt ,t).
\]

\bibliographystyle{alpha}
\bibliography{triangle_gamma}

\begin{thebibliography}{{Sai}17}

\bibitem[AS12]{athana_savvidou}
C.~A. Athanasiadis and C.~Savvidou.
\newblock The local {$h$}-vector of the cluster subdivision of a simplex.
\newblock {\em S\'em. Lothar. Combin.}, 66:Art. B66c, 21, 2011/12.

\bibitem[Ath12]{athana_pacific}
C.~A. Athanasiadis.
\newblock Flag subdivisions and {$\gamma$}-vectors.
\newblock {\em Pacific J. Math.}, 259(2):257--278, 2012.

\bibitem[Ath16]{athana_survey}
C.~A. Athanasiadis.
\newblock A survey of subdivisions and local {$h$}-vectors.
\newblock In {\em The mathematical legacy of {R}ichard {P}. {S}tanley}, pages
  39--51. Amer. Math. Soc., Providence, RI, 2016.

\bibitem[Br{\"a}08]{branden}
P.~Br{\"a}nd{\'e}n.
\newblock Actions on permutations and unimodality of descent polynomials.
\newblock {\em European J. Combin.}, 29(2):514--531, 2008.

\bibitem[Car77]{carlitz}
L.~Carlitz.
\newblock Some expansions and convolution formulas related to {M}ac{M}ahon's
  master theorem.
\newblock {\em SIAM J. Math. Anal.}, 8(2):320--336, 1977.

\bibitem[Cha05]{chapoton_SLC51}
F.~Chapoton.
\newblock Enumerative properties of generalized associahedra.
\newblock {\em S\'em. Lothar. Combin.}, 51:Art. B51b, 16, 2004/05.

\bibitem[Cha06]{chapoton_plein}
F.~Chapoton.
\newblock Sur le nombre de r\'eflexions pleines dans les groupes de {C}oxeter
  finis.
\newblock {\em Bull. Belg. Math. Soc. Simon Stevin}, 13(4):585--596, 2006.

\bibitem[Cha16]{chapoton_stokes}
F.~Chapoton.
\newblock Stokes posets and serpent nests.
\newblock {\em Discrete Math. Theor. Comput. Sci.}, 18(3):Paper No. 18, 30,
  2016.

\bibitem[FZ03a]{FZ2}
S.~Fomin and A.~Zelevinsky.
\newblock Cluster algebras. {II}. {F}inite type classification.
\newblock {\em Invent. Math.}, 154(1):63--121, 2003.

\bibitem[FZ03b]{FZY}
S.~Fomin and A.~Zelevinsky.
\newblock {$Y$}-systems and generalized associahedra.
\newblock {\em Ann. of Math. (2)}, 158(3):977--1018, 2003.

\bibitem[Gal05]{gal}
S.~R. Gal.
\newblock Real root conjecture fails for five- and higher-dimensional spheres.
\newblock {\em Discrete Comput. Geom.}, 34(2):269--284, 2005.

\bibitem[GM17]{garver_mcconville}
A.~{Garver} and T.~{McConville}.
\newblock {Enumerative properties of Grid-Associahedra}.
\newblock {\em ArXiv e-prints}, May 2017.

\bibitem[KS16]{katz_stapledon}
E.~Katz and A.~Stapledon.
\newblock Local {$h$}-polynomials, invariants of subdivisions, and mixed
  {E}hrhart theory.
\newblock {\em Adv. Math.}, 286:181--239, 2016.

\bibitem[Law14]{law}
S.~Law.
\newblock Combinatorial realization of the {H}opf algebra of sashes.
\newblock In {\em 26th {I}nternational {C}onference on {F}ormal {P}ower
  {S}eries and {A}lgebraic {C}ombinatorics ({FPSAC} 2014)}, Discrete Math.
  Theor. Comput. Sci. Proc., AT, pages 621--632. Assoc. Discrete Math. Theor.
  Comput. Sci., Nancy, 2014.

\bibitem[MP17]{pilaud_manneville}
T.~{Manneville} and V.~{Pilaud}.
\newblock {Geometric realizations of the accordion complex of a dissection}.
\newblock {\em ArXiv e-prints}, March 2017.

\bibitem[PPP17]{ppp}
Y.~{Palu}, V.~{Pilaud}, and P.-G. {Plamondon}.
\newblock {Non-kissing complexes and tau-tilting for gentle algebras}.
\newblock {\em ArXiv e-prints}, July 2017.

\bibitem[RS09]{reading_speyer}
N.~Reading and D.~E. Speyer.
\newblock Cambrian fans.
\newblock {\em J. Eur. Math. Soc. (JEMS)}, 11(2):407--447, 2009.

\bibitem[{Sai}17]{saito_F_triangle}
K.~{Saito}.
\newblock The zero loci of {$F$}-triangles.
\newblock {\em Acta Math. Vietnam.}, 42(2):209--236, 2017.

\bibitem[Sta92]{stanley_subdivision}
R.~P. Stanley.
\newblock Subdivisions and local {$h$}-vectors.
\newblock {\em J. Amer. Math. Soc.}, 5(4):805--851, 1992.

\bibitem[Sta96]{stanley_cca}
R.~P. Stanley.
\newblock {\em Combinatorics and commutative algebra}, volume~41 of {\em
  Progress in Mathematics}.
\newblock Birkh\"auser Boston, Inc., Boston, MA, second edition, 1996.

\bibitem[Vis13]{visontai}
M.~Visontai.
\newblock Some remarks on the joint distribution of descents and inverse
  descents.
\newblock {\em Electron. J. Combin.}, 20(1):Paper 52, 12, 2013.

\end{thebibliography}

\end{document}